\documentclass[12pt]{amsart}

\usepackage{amssymb}

\newtheorem{theorem}{Theorem}[section]
\newtheorem{lemma}[theorem]{Lemma}
\newtheorem{proposition}[theorem]{Proposition}
\newtheorem{corollary}[theorem]{Corollary}
\newtheorem{remark}[theorem]{Remark}

\theoremstyle{definition}
\newtheorem{definition}[theorem]{Definition}

\begin{document}
\baselineskip=15pt

\title[Manifolds with trivial tangent bundle]{Principal
bundles on compact complex manifolds with trivial tangent bundle}

\author[I. Biswas]{Indranil Biswas}

\address{School of Mathematics, Tata Institute of Fundamental
Research, Homi Bhabha Road, Bombay 400005, India}

\email{indranil@math.tifr.res.in}



\date{}

\begin{abstract}
Let $G$ be a connected complex Lie group and $\Gamma\, \subset\, G$
a cocompact lattice. Let $H$ be a complex Lie group. We
prove that a holomorphic principal $H$--bundle $E_H$ over $G/\Gamma$
admits a holomorphic connection if and only if $E_H$ is invariant.
If $G$ is simply connected, we show that a holomorphic
principal $H$--bundle $E_H$ over $G/\Gamma$ admits a flat holomorphic
connection if and only if $E_H$ is homogeneous.
\end{abstract}

\maketitle

\section{Introduction}\label{sec1}

Let $T\,=\, {\mathbb C}^n/\Gamma$ be a complex torus, so $\Gamma$
is a lattice of ${\mathbb C}^n$ of maximal rank. For any
$x\, \in\, T$, let $\tau_x\, :\, T\, \longrightarrow\, T$
be the holomorphic automorphism defined by $z\, \longmapsto\, z+x$.
Let $H$ be a connected linear algebraic group defined over
$\mathbb C$. A holomorphic principal $H$--bundle $E_H$ over $T$
admits a holomorphic connection if and only if 
$\tau^*_x E_H$ is holomorphically isomorphic to $E_H$ for every
$x\,\in\, T$; also, if $E_H$ admits a holomorphic connection, then
it admits a flat holomorphic connection
\cite[p. 41, Theorem 4.1]{BG}.

If $\Gamma$ is a cocompact lattice in a
connected complex Lie group $G$, then $G/\Gamma$ is clearly
a compact connected complex manifold with trivial tangent bundle.
Let $M$ be a connected compact complex manifold such that the 
holomorphic tangent bundle $TM$ is holomorphically trivial. Then there 
is a connected complex Lie group $G$ and a cocompact
lattice $\Gamma\, \subset\, G$ such that $G/\Gamma$ is biholomorphic
to $M$. The manifold $M$ is K\"ahler if and only if $M$
is a torus. Our aim here is to investigate principal bundles on $M$
admitting a (flat) holomorphic connection.

Let $G$ be a connected complex Lie group and $\Gamma\, \subset\, G$
a cocompact lattice. For any $g\, \in\, G$, let
$$
\beta_g\, :\, M\, :=\, G/\Gamma \, \longrightarrow\, G/\Gamma
$$
be the automorphism defined by $x\, \longmapsto\, gx$.
Let $H$ be a connected complex Lie group.

A holomorphic principal $H$--bundle $E_H$ over $M$
is called \textit{invariant} if for each $g\, \in\, G$, the pulled
back bundle $\beta^*_g E_H$ is
isomorphic to $E_H$.
A \textit{homogeneous} holomorphic 
principal $H$--bundle on $M$ is a pair $(E_H\, , 
\rho)$, where $f\, :\, E_H\, \longrightarrow\, M$ is a holomorphic 
principal $H$--bundle, and
$$
\rho\, :\, G\times E_H\, \longrightarrow\, E_H
$$
is a holomorphic left--action on the total space of $E_H$,
such that the following two conditions hold:
\begin{enumerate}
\item $(f\circ\rho) (g\, ,z)\, =\, \beta_g(f(z))$
for all $(g\, ,z)\, \in\, G\times E_H$, and

\item the actions of $G$ and $H$ on $E_G$ commute.
\end{enumerate}
If $(E_H\, , \rho)$ is homogeneous, then $E_H$ is invariant.

We prove the following theorem (see Theorem \ref{thm1} and
Proposition \ref{prop1}):

\begin{theorem}\label{t-i}
A holomorphic principal $H$--bundle $E_H$ over $M$ admits a
holomorphic connection if and only if $E_H$ is invariant.

Assume that the group $G$ is simply connected. A holomorphic
principal $H$--bundle $E_H$ over $M$ admits a flat holomorphic
connection if and only if $E_H$ is homogeneous.
\end{theorem}

If $\text{Lie}(G)$ is semisimple and $G$ is simply connected, then we
prove that given any invariant principal $H$--bundle $E_H$, there is
a holomorphic action $\rho\, :\, G\times E_H\, \longrightarrow\, E_H$
such that $(E_H\, ,\rho)$ is homogeneous (see Lemma \ref{lem1}).
This gives the following corollary (see Corollary \ref{cor1}):

\begin{corollary}
Assume that ${\rm Lie}(G)$ is semisimple and $G$ is simply connected.
If a holomorphic principal $H$--bundle $E_H\, \longrightarrow\, M$
admits a holomorphic connection, then it admits a flat
holomorphic connection.
\end{corollary} 

The compact complex manifolds $G/\Gamma$ of the above type with
$G$ non-commutative are the key examples of non--K\"ahler
comapct complex manifolds with trivial canonical bundle.
Recently, these manifolds have started to play important role
in string theory of theoretical physics (see \cite{FIUV},
\cite{BBFTY}, \cite{GP}). They have also become a topic of
investigation in complex differential geometry
(see \cite{GGP}, \cite{Gr}).

\section{Homogeneous bundles and holomorphic connection}

\subsection{Holomorphic connection}

Let $G$ be a connected complex Lie group. Let
$$
\Gamma\, \subset\, G
$$
be a cocompact lattice. So
\begin{equation}\label{e1}
M\, :=\, G/\Gamma
\end{equation}
is a compact complex manifold.
Let $\mathfrak g$ be the Lie algebra of $G$. Using the 
right--invariant vector fields, the holomorphic tangent bundle
$TM$ is identified with the trivial vector bundle
$M\times \mathfrak g$ with fiber $\mathfrak g$, so
\begin{equation}\label{t}
TM\, =\, M\times{\mathfrak g} \, \longrightarrow\, M\, .
\end{equation}
Let
\begin{equation}\label{e2}
\beta\, :\, G\times M \, \longrightarrow\, M
\end{equation}
be the left translation action. The map $\beta$ is holomorphic.
For any $g\,\in\, G$, let
\begin{equation}\label{e01}
\beta_g\, :\, M \, \longrightarrow\, M
\end{equation}
be the automorphism defined by $x\, \longmapsto\, \beta(g\, ,x)$.

Let $H$ be a connected complex Lie group. The Lie algebra of $H$
will be denoted by $\mathfrak h$. We recall that a holomorphic
principal $H$--bundle over $M$ is a complex manifold $E_H$, a
surjective holomorphic submersion $f\, :\, E_H\, \longrightarrow
\,M$ and a right holomorphic action of $H$ on $E_H$
$$
\varphi\, :\, E_H\times H\, \longrightarrow\, E_H
$$
(so $\varphi$ is a holomorphic map), such that the
following two conditions hold:
\begin{enumerate}
\item $f\circ \varphi\, =\, f\circ p_1$, where $p_1$ is the
projection of $E_H\times H$ to the first factor, and

\item the action of $H$ on each fiber of $f$ is free and transitive.
\end{enumerate}

Let $f\, :\, E_H\, \longrightarrow\, M$ be a holomorphic principal
$H$--bundle.
Let $TE_H$ be the holomorphic tangent bundle of
$E_H$. The group $H$ acts on the direct image $f_*TE_H$. The
invariant part
$$
\text{At}(E_H) \, :=\, (f_*TE_H)^H\, \subset\, f_*TE_H
$$
defines a holomorphic vector bundle on $M$, which is called the
\textit{Atiyah bundle} for $E_H$. Let
$$
{\mathcal K} \, :=\, \text{kernel}(df)
$$
be the kernel of the differential $df\,:\, TE_H \,\longrightarrow\,
f^*TM$ of $f$. The invariant direct image $(f_*{\mathcal K})^H$
coincides with the sheaf of sections of the adjoint vector
bundle $\text{ad}(E_H)$. We recall that 
$\text{ad}(E_H)\,\longrightarrow\, M$ is the
vector bundle associated to $E_H$ for the adjoint action of $H$
on $\mathfrak h$. Using the inclusion of $(f_*{\mathcal K})^H$
in $(f_*TE_H)^H$, we get a short exact sequence of holomorphic
vector bundles on $M$
\begin{equation}\label{at}
0\, \longrightarrow\, \text{ad}(E_H)\, \longrightarrow\,
\text{At}(E_H)\, \stackrel{df}{\longrightarrow}\, TM\, 
\longrightarrow\, 0\, .
\end{equation}
It is known as the \textit{Atiyah exact sequence} for $E_H$.
(See \cite{At}.)

A \textit{holomorphic connection} on $E_H$ is a holomorphic splitting
of the short exact sequence in \eqref{at}. In other words, a
holomorphic connection on $E_H$ is a holomorphic homomorphism
$$
D\, :\, TM\, \longrightarrow\, \text{At}(E_H)
$$
such that $(df)\circ D\,=\, \text{Id}_{TM}$, where $df$ is the
homomorphism in \eqref{at} (see \cite{At}).

Let $D$ be a holomorphic connection connection on $E_H$. The
\textit{curvature} of $D$ is the obstruction of $D$ to be
Lie algebra structure preserving (the Lie algebra structure of
sheaves of sections of $TM$ and $\text{At}(E_H)$ is given by
the Lie bracket of vector fields). The curvature of $D$
is a holomorphic section of $\text{ad}(E_H)\otimes \bigwedge^2
(TM)^*$. (See \cite{At} for the details.)

A \textit{flat holomorphic connection} is a holomorphic connection
whose curvature vanishes identically.

\subsection{Invariant and homogeneous bundles}

We will now define invariant holomorphic principal bundles and
homogeneous principal bundles.

\begin{definition}\label{def1}
A holomorphic principal $H$--bundle $E_H$ over $M$ will be
called \textit{invariant} if for each $g\, \in\, G$, the pulled
back holomorphic principal $H$--bundle $\beta^*_g E_H$ is
isomorphic to $E_H$, where $\beta_g$ is the map in \eqref{e01}.
\end{definition}

\begin{definition}\label{def2}
A \textit{homogeneous} holomorphic 
principal $H$--bundle on $M$ is defined to be a pair $(E_H\, , 
\rho)$, where
\begin{itemize}
\item $f\, :\, E_H\, \longrightarrow\, M$ is a holomorphic principal 
$H$--bundle, and

\item $\rho\, :\, G\times E_H\, \longrightarrow\, E_H$ is a 
holomorphic left--action on the total space of $E_H$,
\end{itemize}
such that the following two conditions hold:
\begin{enumerate}
\item $(f\circ\rho) (g\, ,z)\, =\, \beta_g(f(z))$
for all $(g\, ,z)\, \in\, G
\times E_H$, where $\beta_g$ is defined in \eqref{e01}, and

\item the actions of $G$ and $H$ on $E_G$ commute.
\end{enumerate}
\end{definition}

If $(E_H\, , \rho)$ is a homogeneous holomorphic
principal $H$--bundle, then $E_H$ is invariant. Indeed, for
any $g\, \in\, G$, the automorphism of $E_H$ defined by 
$z\, \longmapsto\, \rho(g\, ,z)$ produces an isomorphism of
$E_H$ with $\beta^*_g E_H$.

\section{Automorphisms of principal bundles}
We continue with the notation of the previous section.
We will give a criterion for the existence of a (flat) holomorphic
connection on a holomorphic principal $H$--bundle over $M$.

\begin{theorem}\label{thm1}
A holomorphic principal $H$--bundle $E_H$ over $M$ admits a
holomorphic connection if and only if $E_H$ is invariant.
\end{theorem}

\begin{proof}
Let $f\, :\, E_H\, \longrightarrow\, M$ be a holomorphic principal 
$H$--bundle.
Let $\mathcal A$ denote the space of all pairs of the form
$(g\, ,\phi)$, where $g\, \in\, G$, and
$$
\phi\, :\, E_H\, \longrightarrow\, E_H
$$
is a biholomorphism satisfying the following two conditions:
\begin{enumerate}
\item $\phi$ commutes with the action of $H$ on $E_H$, and

\item $f\circ \phi\,=\, \beta_g\circ f$, where $\beta_g$ is
defined in \eqref{e01}.
\end{enumerate}
So $\phi$ gives a holomorphic isomorphism of the principal $H$--bundle 
$\beta^*_g E_H$ with $E_H$.

We note that $\mathcal A$ is a group using the
composition rule
$$
(g_1\, ,\phi_1)\cdot (g_2\, ,\phi_2)\,=\, (g_1g_2\, ,\phi_1\circ
\phi_2)\, .
$$
In fact, $\mathcal A$ is a complex Lie group, and the Lie algebra
$$
{\mathfrak a} \, :=\, \text{Lie}(\mathcal A)
$$
is identified with
$H^0(M,\, \text{At}(E_H))$. Note that from \eqref{t} it follows that
all holomorphic vector fields on $M$ are given by the Lie algebra of
$G$ using the action $\beta$ in \eqref{e2}. The compactness of $M$
ensures that $\mathcal A$ is of finite dimension. As noted
before, using the Lie bracket of vector fields, the
sheaf of holomorphic sections of $\text{At}(E_H)$ 
has the structure of a Lie algebra. Hence the space $H^0(M,\, 
\text{At}(E_H))$ of global holomorphic sections is a Lie algebra.

Let
\begin{equation}\label{q}
q\, :\, {\mathcal A}\,\longrightarrow\, G
\end{equation}
be the projection defined by $(g\, ,\phi)\, \longmapsto\, g$.
The homomorphism of Lie algebras
\begin{equation}\label{q1}
dq\, :\, {\mathfrak a}\,=\, H^0(M,\,\text{At}(E_H))
\,\longrightarrow\,{\mathfrak g}\, :=\, \text{Lie}(G)
\,=\, H^0(M,\, TM)
\end{equation}
associated to $q$ in \eqref{q} coincides with the one given by
the homomorphism $df$ in \eqref{at}.

First assume that $E_H$ admits a holomorphic connection. Recall
that a holomorphic connection on $E_H$ is a holomorphic splitting
of the exact sequence in \eqref{at}.
Using a holomorphic connection, the vector bundle
$\text{At}(E_H)$ gets identified with the
direct sum $\text{ad}(E_H)\oplus TM$. In particular, the
homomorphism
$$
H^0(M,\,\text{At}(E_H))\,\longrightarrow\, H^0(M,\, TM)
$$
induced by $df$ in \eqref{at} is surjective. Hence the homomorphism
$dq$ in \eqref{q1} is surjective. Since $G$ is connected, this implies
that the homomorphism $q$ in \eqref{q} is surjective. This
immediately implies that $E_H$ is invariant (recall that for any
$(g\, ,\phi)\,\in\, \mathcal A$, the map
$\phi$ is a holomorphic isomorphism of the principal $H$--bundle
$\beta^*_g E_H$ with $E_H$).

To prove the converse, assume that $E_H$ is invariant. Therefore,
the homomorphism $q$ in \eqref{q} is surjective. Hence the
homomorphism $dq$ in \eqref{q1} is surjective. Let $d$ be the dimension
of $G$. Since $dq$ is surjective, and $TM$ is the trivial vector
bundle of rank $d$, there are $d$ sections
$$
\sigma_1\, , \cdots\, ,\sigma_d\, \in\, H^0(M,\,\text{At}(E_H))
$$
such that $H^0(M,\, TM)\,=\, \mathfrak g$ is generated by
$\{dq(\sigma_1)\, ,\cdots\, ,dq(\sigma_d)\}$.

We have a holomorphic homomorphism
$$
D\, :\, TM\, \longrightarrow\, \text{At}(E_H)
$$
defined by $\sum_{i=1}^d c_i\cdot dq(\sigma_i)(x)\, \longmapsto\,
\sum_{i=1}^d c_i\cdot \sigma_i(x)$, where $x\, \in\, M$, and
$c_i\, \in\, \mathbb C$. It is straight--forward to check that
$(df)\circ D\,=\, \text{Id}_{TM}$, where $df$ is the
homomorphism in \eqref{at}. Hence $D$ defines a holomorphic
connection on $E_H$.
\end{proof}

\begin{proposition}\label{prop1}
Assume that the group $G$ is simply connected. A holomorphic
principal $H$--bundle $E_H$ over $M$ admits a flat holomorphic
connection if and only if $E_H$ is homogeneous.
\end{proposition}

\begin{proof}
Let $f\, :\, E_H\, \longrightarrow\, M$ be a holomorphic principal 
$H$--bundle equipped with a flat holomorphic connection $D$. 
Therefore, $E_H$ is given by a homomorphism from the
fundamental group $\pi_1(M,\, (e\, ,e\Gamma))$ to $H$. We will
construct an action of $G$ on $E_H$.

Let $p_2\, :\, G\times M \, \longrightarrow\, M$ be the projection
to the second factor. Since $G$ is simply connected, the homomorphisms
$$
p_{2*}\, , \beta_* \, :\, \pi_1(G\times M,\, (e\, ,e\Gamma)) 
\,\longrightarrow\, \pi_1(M,\,(e\, ,e\Gamma))
$$
induced by $p_2$ and $\beta$ (see \eqref{e2}) coincide. Therefore,
there is a canonical isomorphism of flat principal $H$--bundles
$$
\mu\, :\, p^*_2 E_H\, \longrightarrow\, \beta^* E_H
$$
which is the identity map over $\{e\}\times M$.

This map $\mu$ defines an action
$$
\rho\, :\, G\times E_H\, \longrightarrow\, E_H\, ;
$$
for any $(g\, ,x)\, \in\, G\times M$, the isomorphism
$$
\rho_{g,x}\, :\, (E_H)_x \, \longrightarrow\, (E_H)_{\beta_g(x)}
$$
is the restriction of $\mu$ to $(g\, ,x)$, where $\beta_g$
is the map in \eqref{e01}.
This action $\rho$ makes $E_H$ a homogeneous bundle.

To prove the converse, take a homogeneous holomorphic principal
$H$--bundle $(E_H\, ,\rho)$. For any $g\, \in\, G$, let
$$
\rho_{g}\, :\, E_H \, \longrightarrow\, E_H
$$
be the map defined by $z\, \longmapsto\, \rho(g\, ,z)$.
Consider the group $\mathcal A$ constructed in the proof
of Theorem \ref{thm1}. Let
\begin{equation}\label{delta}
\delta\, :\, G\, \longrightarrow\, \mathcal A
\end{equation}
be the homomorphism defined by $g\, \longmapsto\, (g\, ,\rho_g)$.
It is easy to see that
\begin{equation}\label{cc}
q\circ \delta\, =\, \text{Id}_G\, ,
\end{equation}
where $q$ is the homomorphism in \eqref{q}. 

Let
\begin{equation}\label{delta2}
d\delta\, :\, {\mathfrak g}\, \longrightarrow\, \text{Lie}
({\mathcal A})\,=\, H^0(M,\, \text{At}(E_H))
\end{equation}
be the homomorphism of Lie algebras associated to the homomorphism
$\delta$ in \eqref{delta}. From \eqref{cc} it follows that
\begin{equation}\label{cc2}
(dq)\circ d\delta\,=\, \text{Id}_{\mathfrak g}\, ,
\end{equation}
where $dq$ is the homomorphism in \eqref{q1} (it is the
homomorphism of Lie algebras corresponding to $q$).

Since $TM$ is the trivial vector bundle with fiber $\mathfrak g$,
the homomorphism $d\delta$ in \eqref{delta2} produces a
homomorphism of vector bundles
$$
\widetilde{d\delta}\, :\, TM\, \longrightarrow\, \text{At}(E_H)\, ;
$$
$\widetilde{d\delta}(x\, ,v)\,=\, (d\delta)(v)(x)$,
where $v\, \in\, \mathfrak g$, and $x\, \in\, M$.
Combining \eqref{cc2} and the fact that the
homomorphism $dq$ coincides with the one given by the homomorphism
$df$ in \eqref{at}, we conclude that $(df)\circ\widetilde{d\delta}
\,=\, \text{Id}_{TM}$. Therefore, $\widetilde{d\delta}$ defines a
holomorphic connection on $E_H$. The curvature of
this holomorphic connection vanishes identically because the
linear map $d\delta$ is Lie algebra structure preserving.
\end{proof}

\begin{remark}\label{rem1}
{\rm In Proposition \ref{prop1}, it is essential to assume that
$G$ is simply connected. To give examples, take $G$ to be an
elliptic curve and $\Gamma$ to be the trivial group $e$. Take
$H\, =\, {\mathbb C}^*$, and take $E_H\,=\, L$ to be a nontrivial
holomorphic line 
bundle of degree zero. Then $L$ admits a flat holomorphic connection
because $L$ is topologically trivial, while $L$ does not admit 
any homogeneous structure because $L$ is not trivial.}
\end{remark}

\section{The semisimple case}

In this section we assume that $\mathfrak g$ is semisimple, and
$G$ is simply connected.

Let $f\, :\, E_H\, \longrightarrow\, M\, :=\, G/\Gamma$ be
an invariant holomorphic principal $H$--bundle.

\begin{lemma}\label{lem1}
There is a holomorphic left-action of $G$
$$
\rho\, :\, G\times E_H\, \longrightarrow\, E_H
$$
such that $(E_H\, ,\rho)$ is homogeneous.
\end{lemma}

\begin{proof}
Consider the homomorphism
$$
dq\, :\, {\mathfrak a}\, :=\,
\text{Lie}({\mathcal A})\,\longrightarrow\,{\mathfrak g}
$$
in \eqref{q1}. It is surjective because $E_H$ is invariant implying
that $q$ is surjective. Since ${\mathfrak g}$ is semisimple, there
is a Lie algebra homomorphism
$$
\theta\, :\, {\mathfrak g}\,\longrightarrow\,{\mathfrak a}
$$
such that $(dq)\circ\theta \,=\, \text{Id}_{\mathfrak g}$
(see \cite[p. 91, Corollaire 3]{Bo}). Fix such a homomorphism
$\theta$.

Since $G$ is simply connected, there is a unique
homomorphism of Lie groups
$$
\Theta\, :\, G\,\longrightarrow\, {\mathcal A}
$$
such that $\theta$ is the Lie algebra
homomorphism corresponding to $\Theta$. Now we have an action
$$
\rho\, :\, G\times E_H\, \longrightarrow\, E_H
$$
defined by $\Theta(g) \,=\, (g\, ,\{z\longmapsto \rho(g\, ,z)\})$.
The pair $(E_H\, ,\rho)$ is a homogeneous holomorphic principal 
$H$--bundle.
\end{proof}

Combining Lemma \ref{lem1} with Theorem \ref{thm1} and
Proposition \ref{prop1}, we get the following:

\begin{corollary}\label{cor1}
If a holomorphic principal $H$--bundle $E_H\, \longrightarrow\,
G/\Gamma$ admits a holomorphic connection, then it admits a flat
holomorphic connection.
\end{corollary}


\end{document}